\documentclass[preprint,12pt,3p]{elsarticle}
\usepackage{verbatim}
\usepackage{amsmath,amsthm,amsfonts,amssymb}
\usepackage{algorithm}
\usepackage{algpseudocode}
\usepackage{mathdots}
\usepackage{cases}
\usepackage[hidelinks]{hyperref}
\usepackage{graphics}
\usepackage{graphicx}
\usepackage{caption}
\usepackage{subcaption}
\usepackage{multicol}
\usepackage[normalem]{ulem}
\usepackage{xcolor,sectsty}
\definecolor{astral}{RGB}{46,116,181}
\subsectionfont{\color{astral}}
\sectionfont{\color{astral}}
\newtheorem{theorem}{Theorem}[section]
\newtheorem{lemma}[theorem]{Lemma}

\newtheorem{definition}[theorem]{Definition}
\newtheorem{example}[theorem]{Example}

\newtheorem{thm}{Theorem}[section]

\newtheorem{corollary}[thm]{Corollary}

\definecolor{darkslategray}{rgb}{0.18, 0.31, 0.31}
\definecolor{warmblack}{rgb}{0.0, 0.26, 0.26}
\definecolor{thrdfc}{rgb}{0.36, 0.54, 0.66}
\definecolor{bole}{rgb}{0.55, 0.71, 0.0}
\journal{arXiv.org}

\newcommand{\mc}{\mathcal}
\newcommand{\mb}{\mathbb}
\newcommand{\dg}{{\dagger}}
\newcommand{\m}{{*_M}}
\newcommand{\n}{{*_N}}
\newcommand{\1}{{*_1}}
\newcommand{\2}{{*_2}}
\newcommand{\3}{{*_3}}

\begin{document}

\begin{frontmatter}
\title{ \textcolor{warmblack}{\bf A note on numerical ranges of tensors}}


\author{Nirmal Chandra Rout$^{\dag a}$, Krushnachandra Panigrahy$^{\dag b}$, Debasisha Mishra{$^{\dag c}$}}

\address{$^{\dag}$ Department of Mathematics,\\
National Institute of Technology Raipur,\\
Raipur, Chhattisgarh, India.\\
\textit{E-mail$^a$}: \texttt{nrout89\symbol{'100}gmail.com}\\
 \textit{E-mail$^b$}: \texttt{kcp.224\symbol{'100}gmail.com }\\
\textit{E-mail$^c$}: \texttt{dmishra\symbol{'100}nitrr.ac.in.}}

\begin{abstract}
\textcolor{warmblack}{
Theory of numerical range and numerical radius for tensors is not studied much in the literature. In 2016, Ke {\it et al.}  [Linear Algebra Appl., 508 (2016) 100-132]  introduced first the notion of  numerical range of a tensor via the $k$-mode product. However, the convexity of the numerical range via the $k$-mode product was not proved by them.
In this paper, the notion of numerical range and numerical radius for even-order square tensors using inner product via the Einstein product are introduced first. We provide some sufficient conditions using numerical radius for a tensor to being unitary. The convexity of the numerical range is also proved. We also provide an algorithm to plot the numerical range of a tensor. Furthermore, some properties of the numerical range for the Moore--Penrose inverse of a tensor are discussed.
} 
\end{abstract}


\end{frontmatter}

\section{Introduction}\label{sec1}
The concepts of numerical range and numerical radius have been studied extensively over the last few decades. This is because they are very useful in studying and understanding the role of matrices and operators \cite{bonsall1971, bonsall1973, halmos1967, horn1991} in applications such as numerical analysis and differential equations \cite{axelsson1994, cheng1999, eiermann1993, fiedler1995, goldberg1982, kirkland2001, maroulas2002, psarrakos2002}. The numerical radius is frequently employed as a more reliable indicator of the rate of convergence of iterative methods than the spectral radius \cite{axelsson1994, eiermann1993}. Recently, tensor numerical ranges have been introduced by Ke {\it et al.} \cite{ke2016} on the basis of tensor inner products and tensor norms via $k$-mode product. These have the same properties as those of the numerical ranges of matrices, except the normality, projection, and unitary invariance properties. 

The numerical range is a set of complex numbers associated with a given $n\times n$ matrix $A$:
\begin{equation}\label{eq:nrmtrx}
    W(A)=\{\left\langle Ax,x\right\rangle:\;x\in\mb{C}^{n},\|x\|=1\},
\end{equation}
where $\left\langle x,y\right\rangle=y^{*}x$ for $x,y\in \mb{C}^{n}$ and $\|x\|=\left\langle x,x\right\rangle^{1/2}$. Note that the notion of the numerical range of a matrix is applicable for square matrices and it uses the conjugate transpose. So, to extend the notion of the numerical range of matrices to tensor case, we need a square tensor and the notion of tensor transpose.

Tensors are generalizations of scalars (that have no index), vectors (that have exactly one index), and matrices (that have exactly two indices) to an arbitrary number of indices. An $N^{th}$-order tensor is an element of $\mb{F}^{I_{1}\times \ldots \times I_{N}}$, which is the set of order $N$ complex tensors. Here $I_{1}, I_{2}, \ldots, I_{N}$ are dimensions of the first, second, $\ldots$ , $N^{th}$-mode/way, respectively. The order of a tensor is the number of modes present in it. Thus, a zero-order tensor is a scalar, a first-order tensor is a vector while a second-order tensor is a matrix. Higher-order tensors are tensors of order three or higher. If $N$ is even,  then it is an {\it even-order tensor} otherwise it is an {\it odd-order tensor}. Further, if $N=m$ and $I_{1}=I_{2}=\ldots=I_{m}=n$, then the tensor is said to be $m^{th}$-order $n$-dimensional tensor.

Higher-order tensors are denoted by calligraphic letters like $\mc{A}$. In particular, $a_{ijk}$ denotes an $(i, j, k)^{th}$ element of a third order tensor $\mc{A}$. Different parts of a third-order tensor is shown in the Figure \ref{fig:dfrntprtstnsr1}.
\begin{figure}[h!]
     \centering
     \begin{subfigure}[b]{0.2\textwidth}
         \centering
         \includegraphics[scale=0.8]{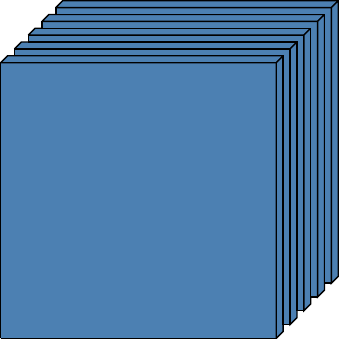}
         \caption{Frontal Slice}
         \label{fig:frntlslc}
     \end{subfigure}
     \hfill
     \begin{subfigure}[b]{0.2\textwidth}
         \centering
         \includegraphics[scale=0.8]{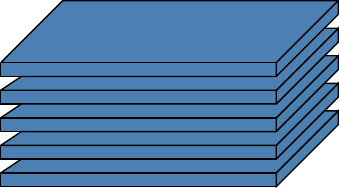}
         \caption{Horizontal Slice}
         \label{fig:hrzntlslc}
     \end{subfigure}
     \hfill
     \begin{subfigure}[b]{0.2\textwidth}
         \centering
         \includegraphics[scale=0.8]{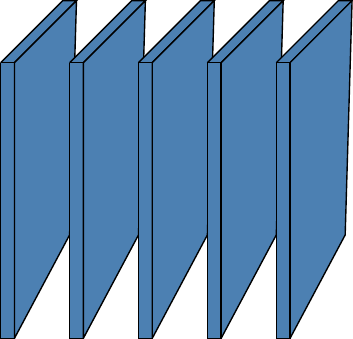}
         \caption{Lateral Slice}
         \label{fig:ltrlslc}
     \end{subfigure}
          \hfill
     \begin{subfigure}[b]{0.2\textwidth}
         \centering
         \includegraphics[scale=0.8]{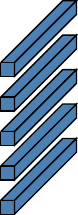}
         \caption{Tube Fiber}
         \label{fig:five over x}
     \end{subfigure}
        \caption{Different parts of a third order tensor}
        \label{fig:dfrntprtstnsr1}
\end{figure}
In particular, consider a third-order tensor of dimension $3\times 3\times 3$ as in Figure \ref{fig:thrdtnsr}. Then, there are three number of frontal slices (see Figure \ref{fig:frntlslc3t}), three number of horizontal slices (see Figure \ref{fig:hrzntlslc3t}), three number of lateral slices (see Figure \ref{fig:ltrlslc3t}) and twenty seven number of tuber fibers (see Figure \ref{fig:hrzntltbfbr3t}).\\
\begin{figure}[h!]
\begin{minipage}{0.5\textwidth}
    \includegraphics[scale=0.8]{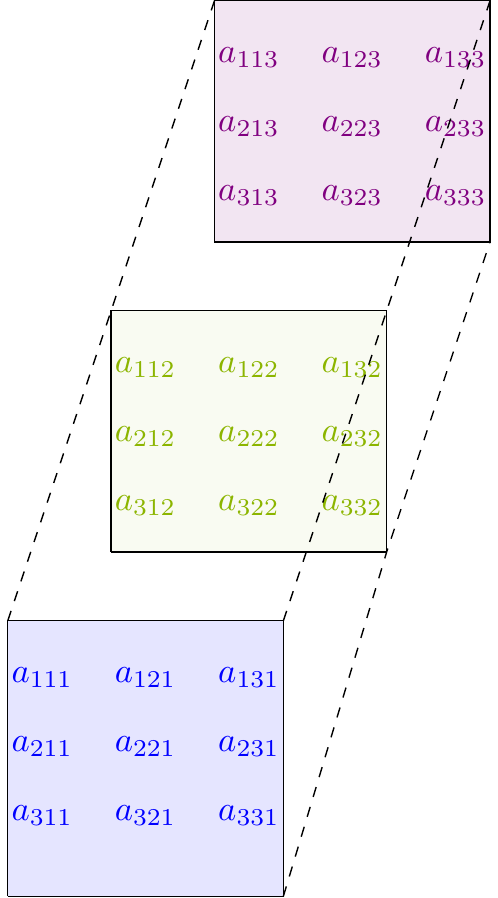}
\end{minipage}
\hspace{-2cm}
\begin{minipage}{0.5\textwidth}
\begin{tabular}{ccc|ccc|ccc}
\hline
    \multicolumn{3}{c}{$\mc{A}(:,:,1)$} & \multicolumn{3}{c}{$\mc{A}(:,:,2)$} & \multicolumn{3}{c}{$\mc{A}(:,:,3)$}  \\
    \hline
    \textcolor{blue}{$a_{111}$} & \textcolor{blue}{$a_{121}$} & \textcolor{blue}{$a_{131}$} & \textcolor{bole}{$a_{112}$} & \textcolor{bole}{$a_{122}$} & \textcolor{bole}{$a_{132}$} & \textcolor{violet}{$a_{113}$} & \textcolor{violet}{$a_{123}$} & \textcolor{violet}{$a_{133}$}\\
    \textcolor{blue}{$a_{211}$} & \textcolor{blue}{$a_{221}$} & \textcolor{blue}{$a_{231}$} & \textcolor{bole}{$a_{212}$} & \textcolor{bole}{$a_{222}$} & \textcolor{bole}{$a_{232}$} & \textcolor{violet}{$a_{213}$} & \textcolor{violet}{$a_{223}$} & \textcolor{violet}{$a_{233}$}\\
    \textcolor{blue}{$a_{311}$} & \textcolor{blue}{$a_{321}$} & \textcolor{blue}{$a_{331}$} & \textcolor{bole}{$a_{312}$} & \textcolor{bole}{$a_{322}$} & \textcolor{bole}{$a_{332}$} & \textcolor{violet}{$a_{313}$} & \textcolor{violet}{$a_{323}$} & \textcolor{violet}{$a_{333}$}\\
    \hline
\end{tabular}
\end{minipage}
    \caption{A third order tensor of dimension $3\times 3\times 3$}
    \label{fig:thrdtnsr}
\end{figure}

\begin{figure}[h!]
     \centering
     \begin{subfigure}[b]{0.3\textwidth}
         \centering
         \includegraphics[width=5.5cm,height=2cm]{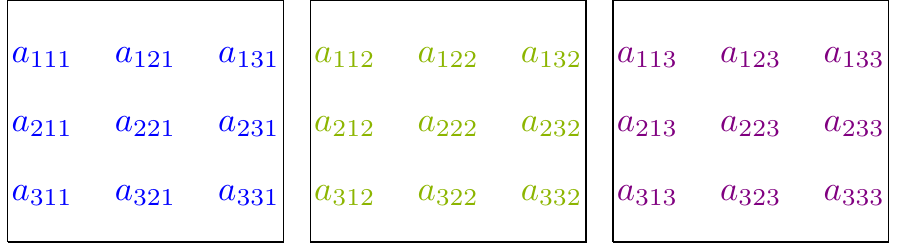}
         \caption{Frontal Slice}
         \label{fig:frntlslc3t}
     \end{subfigure}
     \hfill
     \begin{subfigure}[b]{0.3\textwidth}
         \centering
         \includegraphics[width=5.5cm,height=2cm]{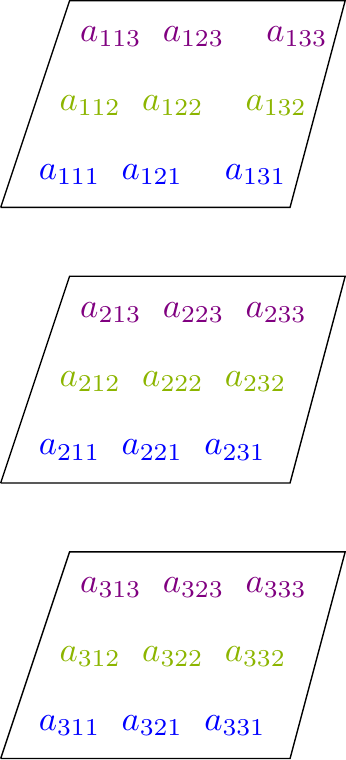}
         \caption{Horizontal Slice}
         \label{fig:hrzntlslc3t}
     \end{subfigure}
     \hfill
     \begin{subfigure}[b]{0.3\textwidth}
         \centering
         \includegraphics[width=5cm,height=2cm]{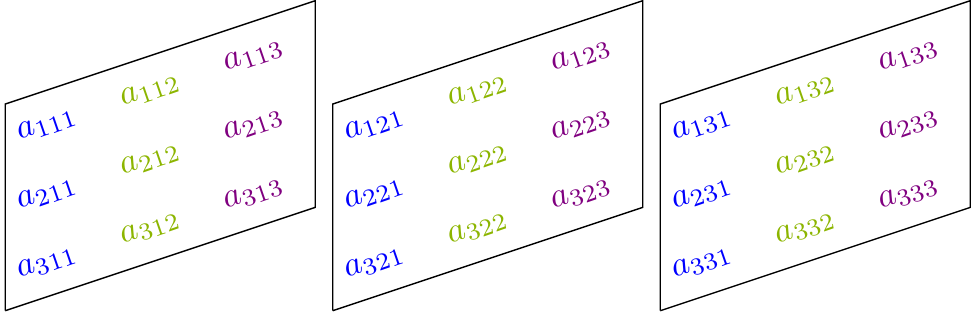}
         \caption{Lateral Slice}
         \label{fig:ltrlslc3t}
     \end{subfigure}
     \hfill
     \begin{subfigure}[b]{0.3\textwidth}
         \centering
         \includegraphics[width=5cm,height=2.5cm]{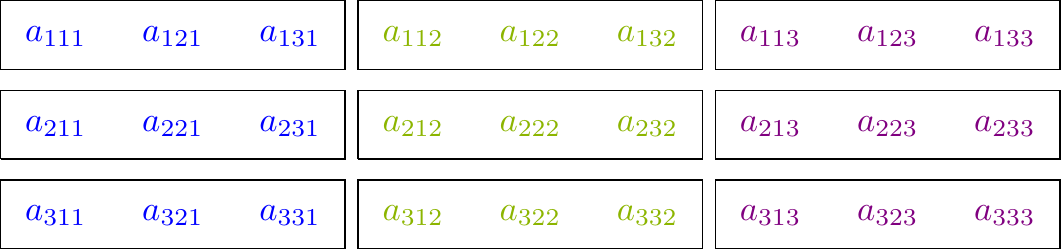}
         \label{fig:frntltbfbr}
     \end{subfigure}
     \begin{subfigure}[b]{0.3\textwidth}
         \centering
    \includegraphics[width=5cm,height=2.5cm]{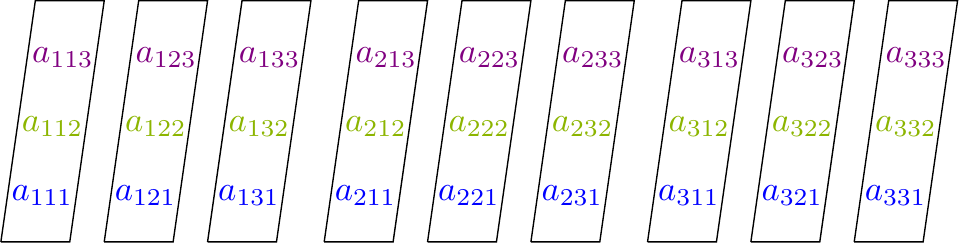}
         \caption{Tube fibers}
         \label{fig:hrzntltbfbr3t}
     \end{subfigure}
     \begin{subfigure}[b]{0.3\textwidth}
         \centering
    \includegraphics[width=5cm,height=2.5cm]{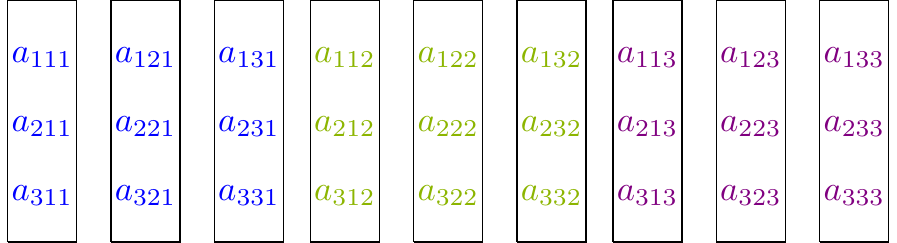}
         \label{fig:ltrltbfbr}
     \end{subfigure}
        \caption{Different parts of a third order tensor of dimension $3\times 3\times 3$}
        \label{fig:dfrntprtstnsr}
\end{figure}

For simplicity, let us denote $I_{1\ldots N}:=I_{1}\times I_{2}\times \ldots \times I_{N}$. The notation $a_{i_{1}\ldots i_{N}}$ (with $1\leq i_{j}\leq I_{j}, ~j = 1,\ldots, N$) represents an $(i_{1},\ldots,i_{N})^{th}$ element of an $N^{th}$-order tensor $\mc{A}\in \mb{F}^{I_{1 \ldots N}}$.  For a tensor $\mc{A}\in \mb{F}^{I_{1 \ldots N}}$, the notation $\mc{A}(:,:,\ldots,:,k),~k=1,2,\ldots,I_{N}$ represents a $(N-1)^{th}$-order tensor in $\mb{F}^{I_{1\ldots (N-1)}}$ which is extracted when the last index is fixed  and is called {\it frontal slice}. A {\it fiber} is identified by fixing each index except one. For a tensor $\mc{A}\in \mb{F}^{I_{1 \ldots N}}$, the notation $\mc{A}(i_{1},i_{2},\ldots,i_{N-1},:)$ represents a $1^{st}$-order tensor in $\mb{F}^{I_{N}}$ which is extracted by fixing each index except the $N^{th}$-index and is called {\it mode-N fiber}. The higher-order analogue of matrix rows and columns are the {\it fibers}.

Let $\mc{A}\in \mb{F}^{I_{1}\times I_{2}\times\ldots \times I_{M}}$ be a tensor and let $\pi$ be a permutation in $S_{M}$ except the identity permutation, where $S_{M}$ represents the permutation group over the set $\{1, 2, \ldots, M\}$, then the $\pi$-transpose of the tensor $\mc{A}$ is defined as
\begin{equation}
    \mc{A}^{T_{\pi}}=(a_{i_{\pi(1)}i_{\pi(2)}\ldots i_{\pi(M)}})\in \mb{F}^{I_{\pi(1)}\times I_{\pi(2)}\times \ldots \times I_{\pi(M)}}.
\end{equation}
Thus, there are $M!-1$ possible transposes associated with the tensor $\mc{A}\in \mb{F}^{I_{1}\times I_{2}\times \ldots \times I_{M}}$. 

\begin{example}
Consider a third-order tensor $\mc{A}\in\mb{R}^{2\times 3\times 4}$ such that 
\begin{center}
  \begin{tabular}{ccc|ccc|ccc|ccc}
\hline
\multicolumn{3}{c}{$\mc{A}(:,:,1)$} & \multicolumn{3}{c}{$\mc{A}(:,:,2)$} & \multicolumn{3}{c}{$\mc{A}(:,:,3)$} &
\multicolumn{3}{c}{$\mc{A}(:,:,4)$} \\
\hline
    \textcolor{blue}{$a_{111}$} & \textcolor{blue}{$a_{121}$} & \textcolor{blue}{$a_{131}$} & \textcolor{brown}{$a_{112}$} & \textcolor{brown}{$a_{122}$} & \textcolor{brown}{$a_{132}$} & \textcolor{violet}{$a_{113}$} & \textcolor{violet}{$a_{123}$} & \textcolor{violet}{$a_{133}$} &
    \textcolor{thrdfc}{$a_{114}$} & \textcolor{thrdfc}{$a_{124}$} & \textcolor{thrdfc}{$a_{134}$} \\     \textcolor{blue}{$a_{211}$} & \textcolor{blue}{$a_{221}$} & \textcolor{blue}{$a_{231}$} & \textcolor{brown}{$a_{212}$} & \textcolor{brown}{$a_{222}$} & \textcolor{brown}{$a_{232}$} & \textcolor{violet}{$a_{213}$} & \textcolor{violet}{$a_{223}$} & \textcolor{violet}{$a_{233}$} &
    \textcolor{thrdfc}{$a_{214}$} & \textcolor{thrdfc}{$a_{224}$} & \textcolor{thrdfc}{$a_{234}$} \\
    \hline
\end{tabular}.
\end{center}
Then, there are $3!-1(=5)$ possible transposes associated with the tensor $\mc{A}\in \mb{R}^{2\times 3\times 4}$. All permutations of three symbols except the identity are,  \begin{equation*}
\pi_{1}=\begin{pmatrix}1&2&3\\2&1&3\end{pmatrix}, \pi_{2}=\begin{pmatrix}1&2&3\\1&3&2\end{pmatrix}, \pi_{3}=\begin{pmatrix}1&2&3\\3&2&1\end{pmatrix},
\pi_{4}=\begin{pmatrix}1&2&3\\2&3&1\end{pmatrix},
\pi_{5}=\begin{pmatrix}1&2&3\\3&1&2\end{pmatrix}.
\end{equation*}
Now, the transpose of the tensor $\mc{A}$ corresponding to the permutation $\pi_{1}$ is $\mc{A}^{T_{\pi_{1}}}\in \mb{R}^{3\times 2\times 4}$ ($({i_{1},i_{2},i_{3}})$-th position element goes to $({i_{\pi_{1}{(1)}},i_{\pi_{1}{(2)}},i_{\pi_{1}{(3)}}})$-th, i.e., $(i_{2},i_{1},i_{3})$-th position):
\begin{center}
    \begin{tabular}{cc|cc|cc|cc}
\hline
\multicolumn{2}{c}{$\mc{A}^{T_{\pi_{1}}}(:,:,1)$} & \multicolumn{2}{c}{$\mc{A}^{T_{\pi_{1}}}(:,:,2)$} & \multicolumn{2}{c}{$\mc{A}^{T_{\pi_{1}}}(:,:,3)$} &
\multicolumn{2}{c}{$\mc{A}^{T_{\pi_{1}}}(:,:,4)$}
\\
        \hline
    \textcolor{blue}{$a_{111}$} & \textcolor{blue}{$a_{211}$} & \textcolor{brown}{$a_{112}$} & \textcolor{brown}{$a_{212}$} & \textcolor{violet}{$a_{113}$} & \textcolor{violet}{$a_{213}$} & \textcolor{thrdfc}{$a_{114}$} & \textcolor{thrdfc}{$a_{214}$} \\ \textcolor{blue}{$a_{121}$} &
    \textcolor{blue}{$a_{221}$} & \textcolor{brown}{$a_{122}$} & \textcolor{brown}{$a_{222}$} & \textcolor{violet}{$a_{132}$} & \textcolor{violet}{$a_{232}$} & \textcolor{thrdfc}{$a_{124}$} & \textcolor{thrdfc}{$a_{224}$} \\ \textcolor{blue}{$a_{131}$} & \textcolor{blue}{$a_{231}$} &
    \textcolor{brown}{$a_{132}$} & \textcolor{brown}{$a_{232}$} & \textcolor{violet}{$a_{133}$} & \textcolor{violet}{$a_{233}$} & \textcolor{thrdfc}{$a_{134}$} & \textcolor{thrdfc}{$a_{234}$}\\
    \hline
\end{tabular}.
\end{center}
Now, the transpose of the tensor $\mc{A}$ corresponding to the permutation $\pi_{2}$ is $\mc{A}^{T_{\pi_{2}}}\in \mb{R}^{2\times 4\times 3}$ ($(i_{1},i_{2},i_{3})$-th position element goes to $(i_{\pi_{2}{(1)}},i_{\pi_{2}{(2)}},i_{\pi_{2}{(3)}})$-th, i.e., $(i_{1},i_{3},i_{2})$-th position):
\begin{center}
    \begin{tabular}{cccc|cccc|cccc}
\hline
    \multicolumn{4}{c}{$\mc{A}^{T_{\pi_{2}}}(:,:,1)$} & \multicolumn{4}{c}{$\mc{A}^{T_{\pi_{2}}}(:,:,2)$} & \multicolumn{4}{c}{$\mc{A}^{T_{\pi_{2}}}(:,:,3)$}  \\
    \hline
    \textcolor{blue}{$a_{111}$} & \textcolor{brown}{$a_{112}$} & \textcolor{violet}{$a_{113}$} & \textcolor{thrdfc}{$a_{114}$} & \textcolor{blue}{$a_{121}$} & \textcolor{brown}{$a_{122}$} & \textcolor{violet}{$a_{123}$} & \textcolor{thrdfc}{$a_{124}$} & \textcolor{blue}{$a_{131}$} &
    \textcolor{brown}{$a_{132}$} & \textcolor{violet}{$a_{133}$} & \textcolor{thrdfc}{$a_{134}$} \\ \textcolor{blue}{$a_{211}$} & \textcolor{brown}{$a_{212}$} & \textcolor{violet}{$a_{231}$} & \textcolor{thrdfc}{$a_{214}$} & \textcolor{blue}{$a_{221}$} & \textcolor{brown}{$a_{222}$} & \textcolor{violet}{$a_{223}$} & \textcolor{thrdfc}{$a_{224}$} & \textcolor{blue}{$a_{231}$} &
    \textcolor{brown}{$a_{232}$} & \textcolor{violet}{$a_{233}$} & \textcolor{thrdfc}{$a_{234}$} \\
    \hline
\end{tabular}.
\end{center}
Similarly, other transposes are as follows. \\ $\mc{A}^{T_{\pi_{3}}}\in \mb{R}^{4\times 3\times 2}$ ($(i_{1},i_{2},i_{3})$-th position element goes to $(i_{\pi_{3}{(1)}},i_{\pi_{3}{(2)}},i_{\pi_{3}{(3)}})$-th, i.e., $(i_{3},i_{2},i_{1})$-th position):
\begin{center}
    \begin{tabular}{ccc|ccc}
\hline
    \multicolumn{3}{c}{$\mc{A}^{T_{\pi_{3}}}(:,:,1)$} & \multicolumn{3}{c}{$\mc{A}^{T_{\pi_{3}}}(:,:,2)$} \\
    \hline
    \textcolor{blue}{$a_{111}$} & \textcolor{blue}{$a_{121}$} & \textcolor{blue}{$a_{131}$} & \textcolor{blue}{$a_{211}$} & \textcolor{blue}{$a_{221}$} & \textcolor{blue}{$a_{231}$} \\
    \textcolor{brown}{$a_{112}$} & \textcolor{brown}{$a_{122}$} & \textcolor{brown}{$a_{132}$} & \textcolor{brown}{$a_{212}$} & \textcolor{brown}{$a_{222}$} & \textcolor{brown}{$a_{232}$}\\     \textcolor{violet}{$a_{113}$} & \textcolor{violet}{$a_{123}$} & \textcolor{violet}{$a_{133}$} & \textcolor{violet}{$a_{213}$} & \textcolor{violet}{$a_{223}$} & \textcolor{violet}{$a_{233}$}\\ \textcolor{thrdfc}{$a_{114}$} & \textcolor{thrdfc}{$a_{124}$} & \textcolor{thrdfc}{$a_{134}$} & \textcolor{thrdfc}{$a_{214}$} & \textcolor{thrdfc}{$a_{224}$} & \textcolor{thrdfc}{$a_{234}$}\\
    \hline
\end{tabular}.
\end{center}
$\mc{A}^{T_{\pi_{4}}}\in \mb{R}^{3\times 4\times 2}$ ($(i_{1},i_{2},i_{3})$-th position element goes to $(i_{\pi_{4}{(1)}},i_{\pi_{4}{(2)}},i_{\pi_{4}{(3)}})$-th, i.e., $(i_{2},i_{3},i_{1})$-th position):
\begin{center}
    \begin{tabular}{cccc|cccc}
\hline
    \multicolumn{4}{c}{$\mc{A}^{T_{\pi_{4}}}(:,:,1)$} & \multicolumn{4}{c}{$\mc{A}^{T_{\pi_{4}}}(:,:,2)$} \\
    \hline
    \textcolor{blue}{$a_{111}$} & \textcolor{brown}{$a_{112}$} & \textcolor{violet}{$a_{113}$} & \textcolor{thrdfc}{$a_{114}$} & \textcolor{blue}{$a_{211}$} & \textcolor{brown}{$a_{212}$} &
    \textcolor{violet}{$a_{213}$} & \textcolor{thrdfc}{$a_{214}$} \\ \textcolor{blue}{$a_{121}$} & \textcolor{brown}{$a_{122}$} & \textcolor{violet}{$a_{123}$} & \textcolor{thrdfc}{$a_{124}$} &     \textcolor{blue}{$a_{221}$} & \textcolor{brown}{$a_{222}$} & \textcolor{violet}{$a_{223}$} & \textcolor{thrdfc}{$a_{224}$} \\ \textcolor{blue}{$a_{131}$} & \textcolor{brown}{$a_{132}$} & \textcolor{violet}{$a_{133}$} & \textcolor{thrdfc}{$a_{134}$} & \textcolor{blue}{$a_{231}$} & \textcolor{brown}{$a_{232}$} & \textcolor{violet}{$a_{233}$} & \textcolor{thrdfc}{$a_{234}$}\\
    \hline
\end{tabular}.
\end{center}
$\mc{A}^{T_{\pi_{5}}}\in \mb{R}^{4\times 2\times 3}$ ($(i_{1},i_{2},i_{3})$-th position element goes to $(i_{\pi_{5}{(1)}},i_{\pi_{5}{(2)}},i_{\pi_{5}{(3)}})$-th, i.e., $(i_{3},i_{1},i_{2})$-th position):
\begin{center}
    \begin{tabular}{cc|cc|cc}
\hline
    \multicolumn{2}{c}{$\mc{A}^{T_{\pi_{5}}}(:,:,1)$} & \multicolumn{2}{c}{$\mc{A}^{T_{\pi_{5}}}(:,:,2)$} & \multicolumn{2}{c}{$\mc{A}^{T_{\pi_{5}}}(:,:,3)$}  \\
    \hline
    \textcolor{blue}{$a_{111}$} & \textcolor{blue}{$a_{211}$} & \textcolor{blue}{$a_{121}$} & \textcolor{blue}{$a_{221}$} & \textcolor{blue}{$a_{131}$} & \textcolor{blue}{$a_{231}$} \\     \textcolor{brown}{$a_{112}$} & \textcolor{brown}{$a_{212}$} & \textcolor{brown}{$a_{122}$} & \textcolor{brown}{$a_{222}$} & \textcolor{brown}{$a_{132}$} & \textcolor{brown}{$a_{232}$} \\     \textcolor{violet}{$a_{113}$} & \textcolor{violet}{$a_{213}$} & \textcolor{violet}{$a_{123}$} & \textcolor{violet}{$a_{223}$} & \textcolor{violet}{$a_{133}$} & \textcolor{violet}{$a_{233}$}\\
    \textcolor{thrdfc}{$a_{114}$} & \textcolor{thrdfc}{$a_{214}$} & \textcolor{thrdfc}{$a_{124}$} & \textcolor{thrdfc}{$a_{224}$} & \textcolor{thrdfc}{$a_{134}$} & \textcolor{thrdfc}{$a_{234}$} \\
    \hline
\end{tabular}.
\end{center}

\end{example}

In particular, for $\mc{A}\in \mb{F}^{I_{1}\times I_{2}\times \ldots\times I_{M}\times J_{1}\times J_{2}\times \ldots \times J_{N}}$ and $\pi\in S_{M+N}$ such that 
$\mc{A}^{T_{\pi}}=(a_{j_{1}j_{2}\ldots j_{N}i_{1}i_{2}\ldots i_{M}})\in \mb{F}^{J_{1}\times J_{2}\times \ldots \times J_{N}\times I_{1}\times I_{2}\times \ldots \times I_{M}}$, then it is simply written as $\mc{A}^{T}$. In this paper, whenever we write $\mc{A}^{H}$ or $\mc{A}^{T}$ for  $M+N$ or $2M$ order tensor, then it is always with respect to partition after $M$-modes. Furthermore, if $\mc{A}\in \mb{F}^{I_{1}\times I_{2}\times \ldots \times I_{M}}$, then $\mc{A}^{T}=(a_{1i_{1}i_{2}\ldots i_{M}})\in \mb{F}^{1\times I_{1}\times I_{2}\times \ldots \times I_{M}}$.

There are two ways to define a square tensor. One when each modes are of equal size, i.e., $n\times n\times \ldots \times n$ and another when the first $N$ modes are repeated in the same order, i.e., $I_{1}\times I_{2}\times \ldots \times I_{N}\times I_{1}\times I_{2}\times \ldots \times I_{N}$. Recently, Ke {\it et al.} \cite{ke2016} have extended the notion of the numerical range of a matrix to the former type of square tensor case. They have considered tensor numerical ranges based on inner products via $k$-mode product which may not be convex in general (see Example 1, \cite{ke2016}). Pakmanesh and Afshin \cite{pakmanesh2021}, continued the same study for even-order tensors.

The main objective of this paper is to study the numerical range of a tensor based on inner product and to study the convexity via Einstein product. The Einstein product \cite{einstein2007} $\mc{A}\n\mc{B} \in \mb{C}^{I_{1\ldots M} \times J_{1\ldots L} }$ of tensors $\mc{A} \in \mb{C}^{I_{1 \ldots M} \times K_{1 \ldots N} }$ and $\mc{B} \in \mb{C}^{K_{1\ldots N} \times J_{1\ldots L} }$   is defined by the operation $\n$ via
\begin{equation*}\label{Eins}
(\mc{A}\n\mc{B})_{i_{1}\ldots i_{M}j_{1}\ldots j_{L}}
=\displaystyle\sum_{k_{1}\ldots k_{N}}a_{{i_{1}\ldots i_{M}}{k_{1}\ldots k_{N}}}b_{{k_{1}\ldots k_{N}}{j_{1}\ldots j_{L}}}.
\end{equation*}
\begin{example}
Consider two third-order tensors $\mc{A}\in\mb{R}^{2\times\textcolor{blue}{3\times 3}}$ and $\mc{B}\in\mb{R}^{\textcolor{blue}{3\times 3}\times 2}$ such that\\[0.2cm]
\begin{minipage}{0.5\textwidth}
\begin{center}
 \begin{tabular}{c c c | c c c | c c c}
\hline
  \multicolumn{3}{c}{$\mc{A}(:,:,1)$} & \multicolumn{3}{c}{$\mc{A}(:,:,2)$} & \multicolumn{3}{c}{$\mc{A}(:,:,3)$} \\
\hline
4 & -5 & 4 & 6 & 3 & 1 & 3 & 2 & 3 \\
1 & 3 & 1 & 2 & 4 & 7 & 2 & 1 & 3\\
\hline
\end{tabular};
\end{center}
\end{minipage}
\begin{minipage}{0.5\textwidth}
\begin{center}
 \begin{tabular}{c c c | c c c}
\hline
  \multicolumn{3}{c}{$\mc{B}(:,:,1)$} & \multicolumn{3}{c}{$\mc{B}(:,:,2)$} \\
\hline
1 & 1 & 4 & 4 & 3 & 1 \\
2 & 4 & 3 &  -4 & 0 & 2 \\
2 & 3 & 1 & 0 & 0 & 1 \\
\hline
\end{tabular}.
\end{center}
\end{minipage}\\[0.2cm]
Then, there are two possible Einstein product between the tensors $\mc{A}$ and $\mc{B}$, namely, $\mc{A}\1\mc{B}\in \mb{R}^{2\times 3\times 3\times 2}$ and $\mc{A}\2\mc{B}\in \mb{R}^{2\times 2}$. Let $\mc{C}=\mc{A}\1\mc{B}$ and $\mc{D}=\mc{A}\2\mc{B}$. Then, the tensors $\mc{C}$ and $\mc{D}$ becomes 

\begin{center}
 \begin{tabular}{c c c | c c c | c c c | c c c | c c c | c c c}
\hline
  \multicolumn{3}{c}{$\mc{C}(:,:,1,1)$} & \multicolumn{3}{c}{$\mc{C}(:,:,2,1)$} & \multicolumn{3}{c}{$\mc{C}(:,:,3,1)$} &
   \multicolumn{3}{c}{$\mc{C}(:,:,1,2)$} & \multicolumn{3}{c}{$\mc{C}(:,:,2,2)$} & \multicolumn{3}{c}{$\mc{C}(:,:,3,2)$} \\
\hline
22 & 5 & 12 & 37 & 13 & 17 & 37 & -9 & 22 & -8 & -32 & 12 & 12 & -15 & 12 & 19 & 3 & 9\\
9 & 13 & 21 & 15 & 22 & 38 & 12 & 25 & 28 & -4 & -4 & -24 & 3 & 9 & 3 & 7 & 12 & 18\\
\hline
\end{tabular};
\end{center}
and 
\begin{center}
 \begin{tabular}{c c}
\hline
  \multicolumn{2}{c}{$\mc{D}(:,:)$} \\
\hline
44 & 64 \\
62 &  5 \\
\hline
\end{tabular}.
\end{center}
\end{example}
The associative law for the Einstein product holds. In the above formula, if $\mc{B} \in \mb{C}^{K_{1\ldots N}}$, then $\mc{A}\n\mc{B} \in \mb{C}^{I_{1 \ldots M}}$ and 
\begin{equation*}
(\mc{A}\n\mc{B})_{i_{1}\ldots i_{M}} = \displaystyle\sum_{k_{1}\ldots k_{N}}
a_{{i_{1}\ldots i_{M}}{k_{1}\ldots k_{N}}}b_{{k_{1}\ldots k_{N}}}.
\end{equation*}
This product is  used in the study of the theory of relativity \cite{einstein2007} and in the area of continuum mechanics \cite{lai2009}. Let $A\in \mb{R}^{m \times n}$ and $B\in \mb{R}^{n \times l}$. Then the Einstein product $\1$  reduces to the standard matrix multiplication as
$$(A\1B)_{ij}= \displaystyle\sum_{k=1}^{n} a_{ik}b_{kj}.$$

Now, a natural question is that why do study numerical range of tensor based on inner products via Einstein product? We next pointed out some of the reasons for this. 
\begin{itemize}
    \item The set of invertible tensors \cite{brazell2013} forms a group under the Einstein product.
    \item The tensor formulation via the Einstein product preserves the low-rank structure in the solution and the right-hand side. Such as, in high-dimensional Poisson problem, the solution and the right-hand side, both represented as $n\times n\times \ldots \times n$ data arrays.
    \item The matrix unfolding of a tensor may give rise to larger bandwidths than the original tensor which increases the number of operations and storage locations. For example, the Laplacian matrices in high dimensions have larger bandwidths than the Laplacian tensors.
\end{itemize}
\noindent 
We refer \cite{huang2020}, to know more advantages of studying theory of tensors via the Einstein product.

We defined the numerical range of tensor via the Einstein product by intending that it may contain the tensor eigenvalues defined in the sense of Definition 2.3 of \cite{liang2019}. In 2019, Liang and Zheng \cite{liang2019}  introduced the notion of eigenvalue of a tensor via the Einstein product as following.
\begin{definition}[Definition 2.3, \cite{liang2019}]\label{egndfn}\leavevmode\\
Let $\mc{A}\in \mb{C}^{I_{1 \ldots N}\times I_{1 \ldots N}}$ be a given tensor. If a complex number $\lambda$ and a nonzero tensor $\mc{X}\in \mb{C}^{I_{1 \ldots N}}$ satisfy 
    \begin{equation} 
    \mc{A}\n \mc{X}=\lambda \mc{X},
    \end{equation} 
then we say that $\lambda$ is an eigenvalue of $\mc{A}$, and $\mc{X}$ is the eigentensor with respect to $\lambda$.
\end{definition}
The set of all the eigenvalues of $\mc{A}$ is denoted by $\sigma(\mc{A})$. The spectral radius of tensor $\mc{A}$, denoted by $\rho(\mc{A})$, is defined as
$\rho(\mc{A})=\max\{|\lambda|: \lambda \text{ is an eigenvalue of } \mc{A}\}.$ 

As in case of matrix, the numerical ranges contain the eigenvalues, so the study of the numerical ranges is useful in designing fast algorithms for the calculation of its eigenvalues. The reason for this is revealed in the following well known theorem and corollary (for proofs and discussion see \cite{berger1965, goldberg1982, halmos1967, horn1991, pearcy1960}). 

\begin{theorem}
Let $w(A)=\max\{|z|:\,z\in W(A)\}$ for a matrix $ A\in\mathcal{M}_n(\mb{C})$. Then,
\begin{enumerate}[i)]
    \item $(1/2) \Vert A \Vert _2 \le w(A) \le \Vert A \Vert _2 ,\,$ where $ \Vert \cdot \Vert _2$ denotes the spectral (operator) norm induced on $ \mathcal{M}_n (\mb{C})$.
    \item For every positive integer $m$, $ w(A^{m}) \le w(A)^{m}$.
\end{enumerate}
\end{theorem}

\begin{corollary}
Let $w(A)=\max\{|z|:\,z\in W(A)\}$ for a matrix $ A\in\mathcal{M}_n(\mb{C})$. Then, for any positive integer $m$,
\begin{equation*}
    w(A^m)^{1/m} \le \Vert A^m \Vert _2^{1/m} \le 2^{1/m} w(A^m)^{1/m} \le 2^{1/m} w(A).
\end{equation*}
\end{corollary}
The second objective of this paper is to develop algorithms to compute  the numerical ranges of tensors, which may be useful in designing faster algorithms for the calculation of its eigenvalues.

In 2016, Sun { \it et al.} \cite{sun2016} introduced formally a generalized inverse called the {\it Moore--Penrose inverse of an even-order tensor} via the Einstein product. The authors \cite{sun2016} then used the Moore--Penrose inverse to find the minimum-norm least-squares solution of some multilinear systems. Panigrahy and Mishra \cite{panigrahy2018}, Stanimirovic {\it et al.} \cite{stanimirovic2018},  and  Liang and Zheng \cite{liang2019} independently improved the definition of the Moore--Penrose inverse of an even-order tensor to a tensor of any order via the same product. The definition of the Moore--Penrose inverse of an arbitrary order tensor is recalled below.
\begin{definition}(Definition 1.1, \cite{panigrahy2018})\label{def:defmpi}\leavevmode\\
Let $\mc{A} \in \mb{R}^{I_{1 \ldots N} \times J_{1 \ldots M}}$. The tensor $\mc{X} \in \mb{R}^{J_{1 \ldots M} \times I_{1 \ldots N}}$ satisfying the following four tensor equations:
\begin{eqnarray}
\mc{A}\m\mc{X}\n\mc{A} &=& \mc{A};\label{mpeq1}\\
\mc{X}\n\mc{A}\m\mc{X} &=& \mc{X};\label{mpeq2}\\
(\mc{A}\m\mc{X})^{H} &=& \mc{A}\m\mc{X};\label{mpeq3}\\
(\mc{X}\n\mc{A})^{H} &=& \mc{X}\n\mc{A},\label{mpeq4}
\end{eqnarray}
is defined as the \textbf{Moore--Penrose inverse} of $\mc{A}$, and is denoted by $\mc{A}^{\dg}$.
\end{definition}

The third objective of this paper is to investigate the properties of numerical range of the Moore--Penrose inverse of a tensor.

The rest of this paper is structured as follows. In Section \ref{sec: nrtnsr}, the notion of numerical range of a tensor is introduced, based on inner product via the Einstein product. Also the convexity of numerical range has been verified. In Section \ref{sec: algthmnrtnsr}, an algorithm to plot boundary of the numerical range of a tensor is derived. In Section \ref{sec:nrtnsr}, the notion of numerical radius is used to verify the unitary property of a tensor.. In Section \ref{sec:nrmpitnsr}, some properties of the numerical range of the Moore--Penrose inverse is given.

\section{Numerical range of a tensor}\label{sec: nrtnsr}
In this section, a possible extension of numerical range for a tensor via the Einstein product is introduced. Spectral containment and convexity are two important requirements in the generalization. 

For two tensors $\mc{X},\mc{Y}\in\mb{C}^{I_{1\ldots N}}$, we define an inner product $\left\langle\mc{X},\mc{Y}\right\rangle=\mc{Y}^{H}\n\mc{X}$ and a norm induced by this inner product as $\|\mc{X}\|=\left\langle\mc{X},\mc{X}\right\rangle^{1/2}$. A tensor $\mc{X}\in\mb{C}^{I_{1\ldots N}}$ is said to be a {\it unit tensor}, if $\|\mc{X}\|=1$. According to \eqref{eq:nrmtrx}, it is natural to consider the following generalization.
\begin{definition}\label{def:nrtnsr}
  The numerical range of an even-order square tensor $\mc{A}\in \mb{C}^{I_{1\ldots N}\times I_{1\ldots N}}$, denote it by $W(\mc{A})$, is defined as
\begin{eqnarray}\label{nrtnsrdfn}
    W(\mc{A})
    &=&\{\left\langle \mc{A}\n\mc{X},\mc{X}\right\rangle:\;\mc{X}\text{ is a unit tensor in }\mb{C}^{I_{1\ldots N}}\}.
\end{eqnarray}
\end{definition}
With some elementary calculation it can be shown that,
\begin{equation}\label{eqn:nraltdfn}
    W(\mc{A})=\left\{\frac{\left\langle \mc{A}\n\mc{X},\mc{X}\right\rangle}{\|\mc{X}\|^{2}}:\;\mc{O}\neq\mc{X}\in\mb{C}^{I_{1\ldots N}}\right\}.
\end{equation} 
Note that, in the above Definition \ref{def:nrtnsr} when $N=1$, it coincides with the matrix numerical range defined in \eqref{eq:nrmtrx}. The spectrum of a tensor is always contained in it's numerical range is shown in the next theorem.
\begin{theorem}\label{thm:subset}
  The spectrum of $\mc{A}$ always lies in $ W(\mc{A})$, i.e., $\sigma(\mc{A})\subseteq W(\mc{A})$.
\end{theorem}
Note that when $\mc{A}=\alpha \mc{I}$, we have $\sigma(\mc{A})=\{\alpha\}=W(\mc{A})$. The following example shows that $ W(\mc{A})$ contains some elements which is not in $\sigma(\mc{A})$.
\begin{example}
Consider a tensor $\mc{A}\in \mb{C}^{3\times2\times 3\times 2}$ such that
\begin{center}
{\scriptsize\begin{tabular}{cc|cc|cc|cc|cc|cc}
\hline
\multicolumn{2}{c}{$\mc{A}(:,:,1,1)$} & \multicolumn{2}{c}{$\mc{A}(:,:,1,2)$} & \multicolumn{2}{c}{$\mc{A}(:,:,2,1)$} & \multicolumn{2}{c}{$\mc{A}(:,:,2,2)$} & \multicolumn{2}{c}{$\mc{A}(:,:,3,1)$} & \multicolumn{2}{c}{$\mc{A}(:,:,3,2)$} \\
\hline
    2 & 0 & 0 & 1 & 0 & 0 & 0 & 0 & 0 & 0 & 0 & 0 \\
    0 & 0 & 0 & 0 & 3 & 0 & 0 & -1 & 0 & 0 & 0 & 0 \\
    0 & 0 & 0 & 0 & 0 & 0 & 0 & 0 & 8 & 0 & 0 & 9 \\
\hline
\end{tabular}}.
\end{center}
Here, $\sigma(\mc{A})=\{-1,1,2,3,8,9\}$. Let $\mc{X}=\begin{bmatrix}1/\sqrt{6}&1/\sqrt{6}\\1/\sqrt{6}&1/\sqrt{6}\\1/\sqrt{6}&1/\sqrt{6}\end{bmatrix}\in\mb{R}^{3\times 2}$. Then, $\|\mc{X}\|=1$ and hence $\left\langle \mc{A}\n\mc{X},\mc{X}\right\rangle=6\in W(\mc{A})$. But, $6\notin \sigma(\mc{A})$.
\end{example}
\begin{theorem}\label{thm:lnrprprty}
  Let $ \mc{A}\in\mb{C}^{I_{1\ldots N}\times I_{1\ldots N}}$ and $ \alpha,\beta \in \mb{C}$. Then $ W( \alpha \mc{A} + \beta \mc{I}) = \alpha W(\mc{A}) + \beta $. 
\end{theorem}
Next results shows that the numerical range of sum of two tensors $\mc{A}$ and $\mc{B}$ is always contained in sum of numerical ranges of the individual sets, i.e., $ W(\mc{A}+\mc{B}) \subseteq W(\mc{A}) + W(\mc{B})$.
\begin{theorem}
Let $\mc{A},\mc{B} \in \mb{C}^{I_{1\ldots N}\times I_{1\ldots N}}$. Then $ W(\mc{A}+\mc{B}) \subseteq W(\mc{A}) + W(\mc{B})$.
\end{theorem}
The real part  and imaginary part of the numerical ranges of a tensor are same as the numerical range of the Hermitian part and skew-Hermitian part of that tensor, respectively. This is verified in the next theorem. `Re' and `Im' are used to denote the real and imaginary parts of a set, respectively.
\begin{theorem}
If $ H(\mc{A}) = ( \mc{A} + \mc{A}^{H} ) / 2$ and $ S(\mc{A}) = ( \mc{A} - \mc{A}^{H} ) / 2 $ are the Hermitian part and the skew-Hermitian part of $ 
\mc{A}\in\mb{C}^{I_{1\ldots N}\times I_{1\ldots N}}$, respectively, then
$\displaystyle \textup{Re}\, W(\mc{A}) = W(H(\mc{A})) $   and $\displaystyle  \textup{Im}\, W(\mc{A}) = W(S(\mc{A}))$.
\end{theorem}
The numerical range of a tensor remains unaltered after taking it's transpose. However, the numerical ranges of conjugate transpose of a tensor is equal to that of conjugate of the tensor, which is further equal to conjugate of the numerical range of the tensor. Next, we prove this as a theorem.
\begin{theorem}
Let $ \mc{A}\in\mb{C}^{I_{1\ldots N}\times I_{1\ldots N}}$. Then $W(\mc{A}^{T})=W(\mc{A})$ and $ W(\mc{A}^{H}) = W(\overline{\mc{A}}) = \overline{W(\mc{A})}$.  
\end{theorem}

\begin{theorem}\label{untrythm}
Let $\mc{B}\in \mb{C}^{I_{1\ldots M}\times J_{1\ldots N}}$  complex tensor such that $ \mc{B}^{H}\m\mc{B} = \mc{I}$. Then for any $ \mc{A}\in\mb{C}^{I_{1\ldots M}\times I_{1\ldots M}}$, we get $ W(\mc{B}^{H}\m\mc{A}\m\mc{B}) \subseteq W(\mc{A})$. Equality holds, if $M=N$ and $(I_{1},\ldots,I_{M})=(J_{1},\ldots, J_{N})$, i.e., $\mc{B}$ is unitary.
\end{theorem}
Recall the following well known result for a continuous function.
\begin{theorem}
    Let $f:A\rightarrow X\times Y$ be given by the equation $f(a)=(f_{1}(a),f_{2}(a))$. Then $f$ is continuous if, and only if, the functions $f_{1}:A\rightarrow X$ and $f_{2}:A\rightarrow Y$ are continuous.
\end{theorem}
Next, we recall the notion of a path and path connected in a space.
\begin{definition}
  Given points $x$ and $y$ of the space $X$, a {\it path} in $X$ from $x$ to $y$ is a continuous map $f:[a,b]\rightarrow X$ of some closed interval in the real line into $X$, such that $f(a)=x$ and $f(b)=y$. A space $X$ is said to be path connected if every pair of points of $X$ can be joined by a path in $X$.
\end{definition}
The next lemma is about construction of a path connected set associated with an element of numerical range of a tensor.
\begin{lemma}\label{lma:pthcnctd}
Let $\mc{A}\in \mb{C}^{I_{1\ldots N}\times I_{1\ldots N}}$ be a Hermitian tensor, i.e., $\mc{A}^{H}=\mc{A}$. Also, let $T_{\mc{A}}(\alpha)=\{\mc{X}\text{ is a unit tensor in }\mb{C}^{I_{1\ldots N}}\mid~\left\langle\mc{A}\n\mc{X},\mc{X}\right\rangle=\alpha\}$. Then $T_{\mc{A}}(\alpha)$ is path connected for $\alpha\in W(\mc{A})$.
\end{lemma}
The next result verifies that the numerical range of a tensor defined in Definition \ref{def:nrtnsr} is a convex set.
\begin{theorem}
  For a tensor $\mc{A}\in\mb{C}^{I_{1\ldots N}\times I_{1\ldots N}}$, the numerical range $W(\mc{A})$ is convex.
\end{theorem}
\section{An algorithm for finding tensor numerical range}\label{sec: algthmnrtnsr}
In this section, we provide an algorithm to plot the boundary of the numerical range of a tensor. 
The set of tensors $\{\mc{U}_{1},\mc{U}_{2},\ldots,\mc{U}_{n}\}$, where $\mc{U}_{i}\in\mb{C}^{I_{1\ldots N}}$, is called orthogonal if for $i\neq j$, $\left\langle\mc{U}_{i},\mc{U}_{j}\right\rangle=0$, and further if it satisfies $\left\langle\mc{U}_{i},\mc{U}_{i}\right\rangle=1$ then it is called orthonormal.
\begin{theorem}\label{thm:algm}
Let $\mc{A}\in \mb{C}^{I_{1\ldots N}\times I_{1\ldots N}}$ and $\mc{X}\in\mb{C}^{I_{1\ldots N}}$. Then the following are equivalent.
\begin{enumerate}[i)]
    \item $Re\left\langle\mc{A}\n\mc{X},\mc{X}\right\rangle=\max\{Re(z):\, z\in W(\mc{A})\}$.
    \item $\left\langle H(\mc{A})\n\mc{X},\mc{X}\right\rangle=\max\{r:\,r\in W(H(\mc{A}))\}$.
    \item $H(\mc{A})\n\mc{X}=\lambda_{\max}\mc{X}$, where $\lambda_{\max}=\max\{\lambda: \lambda \in \sigma(H(\mc{A}))\}$.
\end{enumerate}
\end{theorem}
Next, an immediate consequence of the above Theorem \ref{thm:algm} is presented as a corollary without proof.
\begin{corollary}\label{thm_basicalg}
Let $\mc{A}\in \mb{C}^{I_{1\ldots N}\times I_{1\ldots N}}$. Then, 
\begin{equation*}
    \max\{Re(z):\, z\in W(\mc{A})\}=\max\{r:\,r\in W(H(\mc{A}))\}=\max\{\lambda: \lambda \in \sigma(H(\mc{A}))\}.
\end{equation*}
\end{corollary}

Note that, according to Theorem \ref{thm:lnrprprty}, we have $e^{-i\theta}W(e^{i\theta}\mc{A})=W(\mc{A})$ for all $0\leq \theta\leq 2\pi$. 
\begin{theorem}
Let $\mc{A}\in \mb{C}^{I_{1\ldots N}\times I_{1\ldots N}}$ and $\mc{X}_{\theta}$ be the normalized eigentensor corresponding to the maximum eigenvalue of $H(e^{i\theta}\mc{A})$ for some $\theta\in [0,2\pi]$. Then, the complex number $\mc{X}_{\theta}^{H}\n\mc{A}\n\mc{X}_{\theta}=\left\langle\mc{A}\n\mc{X}_{\theta},\mc{X}_{\theta}\right\rangle$ is a boundary point of $W(\mc{A})$.
\end{theorem}
Based on the above theory we next present an algorithm to plot the boundary of the numerical range of a tensor.\\
\begin{tabular}{lll}
\multicolumn{3}{l}{\bf Algorithm 1}\\
    &{\bf Step 1.}& Choose $\theta\in[0,2\pi]$\\
    &{\bf Step 2.} & Calculate $\mc{T}=e^{i\theta}\mc{A}$, for a given $\mc{A}\in\mb{C}^{I_{1\ldots N}\times I_{1\ldots N}}$\\
    &{\bf Step 3.}& $\lambda_{max}= \max\{\lambda:\,\lambda\in\sigma(H(\mc{T}))\}$ \\
    &{\bf Step 4.} & Calculate the normalized eigentensor corresponding to $\lambda_{\max}$, $\mc{X}_{\max}$ (say) \\
    &{\bf Step 5.}& Calculate $z=\left\langle \mc{A}\n \mc{X}_{\max},\mc{X}_{\max}\right\rangle$\\
    &{\bf Step 6.}& Plot $z$\\
    &{\bf Step 7.}& Repeat {\bf Step 1.} to {\bf Step 6.} 
\end{tabular}\\[0.2cm]
Next, we present few numerical examples to illustrate the algorithm introduced above.
\begin{example}
Consider a tensor $\mc{A}\in \mb{C}^{3\times2\times 3\times 2}$ such that
\begin{center}
{\scriptsize\begin{tabular}{cc|cc|cc|cc|cc|cc}
\hline
\multicolumn{2}{c}{$\mc{A}(:,:,1,1)$} & \multicolumn{2}{c}{$\mc{A}(:,:,1,2)$} & \multicolumn{2}{c}{$\mc{A}(:,:,2,1)$} & \multicolumn{2}{c}{$\mc{A}(:,:,2,2)$} & \multicolumn{2}{c}{$\mc{A}(:,:,3,1)$} & \multicolumn{2}{c}{$\mc{A}(:,:,3,2)$} \\
\hline
    0 & 1 & 0 & 1 & 1 & 1 & 0 & 1 & 1 & 1 & 0 & 1  \\
    2 & 1 & 1 & 1 & 1 & 2 & -1 & 1 & 1 & 1 & 0 & 1 \\
    2 & 1 & 1 & 1 & 1 & 2 & -1 & 1 & 1 & 1 & 0 & 1 \\
\hline
\end{tabular}}.
\end{center}
Now, applying the Algorithm 1 to the tensor $\mc{A}$ for 500 different choices of $\theta$ we obtain the Figure \ref{fig:nrfig1}. Also, the eigenvalues of the tensor are plotted (highlighted by `\textcolor{red}{*}'), all of which lies inside the boundary of numerical range of the tensor.
\begin{figure}[h!]
    \centering
    \includegraphics[scale=0.5]{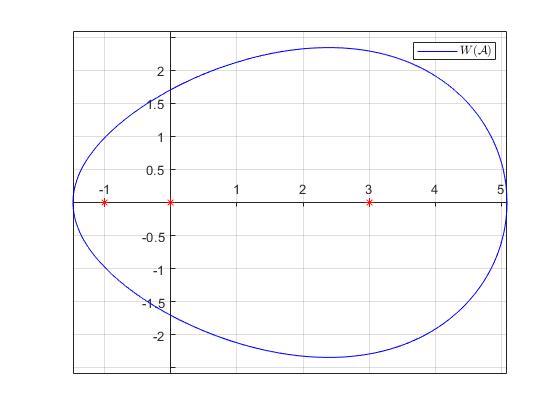}
    \caption{Numerical range for the tensor $\mc{A}$ (the `\textcolor{red}{*}' symbols in red represent the eigenvalues of $\mc{A}$)}
    \label{fig:nrfig1}
\end{figure}
\end{example}
\begin{example}
Consider a tensor $\mc{A}\in \mb{C}^{2\times2\times 2\times 2}$ such that
\begin{center}
{\scriptsize\begin{tabular}{cc|cc|cc|cc}
\hline
\multicolumn{2}{c}{$\mc{A}(:,:,1,1)$} & \multicolumn{2}{c}{$\mc{A}(:,:,2,1)$} & \multicolumn{2}{c}{$\mc{A}(:,:,1,2)$} & \multicolumn{2}{c}{$\mc{A}(:,:,2,2)$} \\
\hline
    0 & 0 & 1-i &1-i & 0& 0 & 0 & 1 +i \\
    1+i & 0 & 0 & 0 & 1-i & 1-i & 0 & 0  \\
\hline
\end{tabular}}.
\end{center}
Now, applying the Algorithm 1 to the tensor $\mc{A}$ for 500 different choices of $\theta$ we obtain the Figure \ref{fig:nrfig2}. Each eigenvalue of the tensor (highlighted by `\textcolor{red}{*}') lies inside the boundary of numerical range of the tensor.
\begin{figure}[h!]
    \centering
    \includegraphics[scale=0.5]{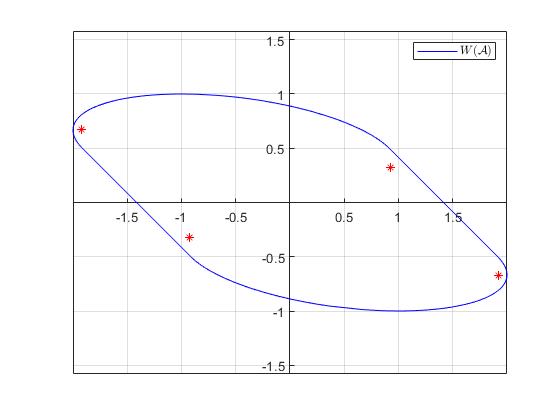}
    \caption{Numerical range for the tensor $\mc{A}$ (the `\textcolor{red}{*}' symbols in red represent the eigenvalues of $\mc{A}$)}
    \label{fig:nrfig2}
\end{figure}
\end{example}
\begin{example}
Consider a tensor $\mc{A}\in \mb{C}^{2\times2\times 2\times 2}$ such that
\begin{center}
{\scriptsize\begin{tabular}{cc|cc|cc|cc}
\hline
\multicolumn{2}{c}{$\mc{A}(:,:,1,1)$} & \multicolumn{2}{c}{$\mc{A}(:,:,2,1)$} & \multicolumn{2}{c}{$\mc{A}(:,:,1,2)$} & \multicolumn{2}{c}{$\mc{A}(:,:,2,2)$} \\
\hline
    i & 0 & 0 & 0 & 0 & 1+i & 0 & 0 \\
    0 & 0 & 1 & 0 & 0 & 0 & 0 & 2+i  \\
\hline
\end{tabular}}.
\end{center}
Now, applying the Algorithm 1 to the tensor $\mc{A}$ for 500 different choices of $\theta$ we obtain the Figure \ref{fig:nrfig3}. All the eigenvalues (highlighted by `\textcolor{red}{*}') of the tensor are on the boundary of the numerical range of the tensor.
\begin{figure}[h!]
    \centering
    \includegraphics[scale=0.5]{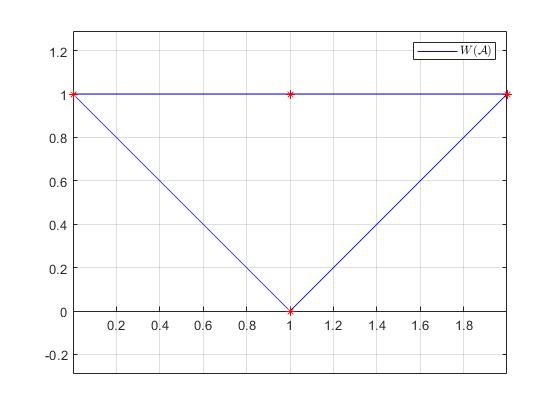}
    \caption{Numerical range for the tensor $\mc{A}$ (the `\textcolor{red}{*}' symbols in red represent the eigenvalues of $\mc{A}$)}
    \label{fig:nrfig3}
\end{figure}
\end{example}
\begin{example}
Consider a Hermitian tensor $\mc{A}\in \mb{C}^{2\times2\times 2\times 2}$ such that
\begin{center}
{\scriptsize\begin{tabular}{cc|cc|cc|cc}
\hline
\multicolumn{2}{c}{$\mc{A}(:,:,1,1)$} & \multicolumn{2}{c}{$\mc{A}(:,:,2,1)$} & \multicolumn{2}{c}{$\mc{A}(:,:,1,2)$} & \multicolumn{2}{c}{$\mc{A}(:,:,2,2)$} \\
\hline
    1 & -3i & i & 1-i & 3i & 4 & 2+5i & 7-i \\
    -i & 2-5i & 1 & 3+i & 1+i & 7+i & 3-i & 0  \\
\hline
\end{tabular}}.
\end{center}
Now, applying the Algorithm 1 to the tensor $\mc{A}$ for 500 different choices of $\theta$ we obtain the Figure \ref{fig:nrfig2}. Each eigenvalue of the tensor (highlighted by `\textcolor{red}{*}') lies inside the boundary of numerical range of the tensor.
\begin{figure}[h!]
    \centering
    \includegraphics[scale=0.5]{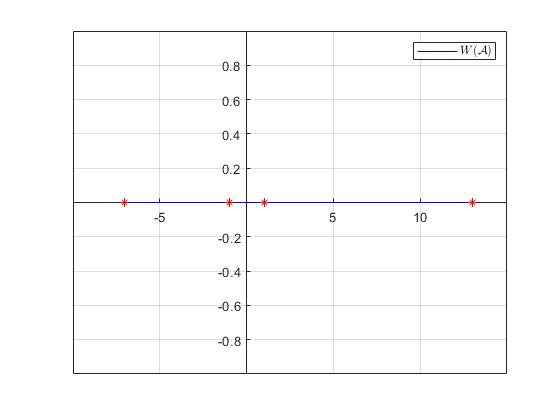}
    \caption{Numerical range for the tensor $\mc{A}$ (the `\textcolor{red}{*}' symbols in red represent the eigenvalues of $\mc{A}$)}
    \label{fig:nrfig4}
\end{figure}
\end{example}
\section{Numerical radius of a tensor}\label{sec:nrtnsr}
In this section, we introduce the notion of numerical radius for tensors and investigate its properties. 
\begin{definition}\label{def:nr}
  The numerical radius of an even-order square tensor $\mc{A}\in\mb{C}^{I_{1\ldots N}\times I_{1\ldots N}}$ is denoted as $w(\mc{A})$ and is defined as
  \begin{equation}
      w(\mc{A})=\max\{|z|\mid z\in W(\mc{A})\}.
  \end{equation}
\end{definition}
Tensors satisfy the Cauchy-Schwarz inequality. Next we state this without proof. One can follow the steps of existing proofs in the literature to verify the inequality.
\begin{theorem}
Let $\mc{A},\mc{B}\in\mb{C}^{I_{1\ldots N}}$. Then $|\langle \mc{A},\mc{B}\rangle|\leq \|\mc{A}\|\|\mc{B}\|$.
\end{theorem}
Now we state a very popular theorem of numerical radius inequality for tensor which may contribute to develop new theories of numerical radius for tensors.
For $\mc{A}\in \mb{C}^{I_{1\ldots N}\times I_{1\ldots N}}$ and $\mc{X}, \mc{Y}\in \mb{C}^{I_{1\ldots N}}$ we define $\|\mc{A}\|=\sup\{\|\mc{A}\n\mc{X}\|:\|\mc{X}\|=1\}=\sup\{\left|\langle\mc{A}\n\mc{X},\mc{Y}\rangle\right|:\|\mc{X}\|=\|\mc{Y}\|=1\}$.
\begin{theorem}\label{thm:bscnrineq}
Let $\mc{A}\in \mb{C}^{I_{1\ldots N}\times I_{1\ldots N}}$, then
$ \dfrac{1}{2}\|\mc{A}\|\leq w(\mc{A})\leq \|\mc{A}\|.$
\end{theorem}
Next we present a corollary as an immediate consequence of Theorem \ref{thm:bscnrineq} without proof.
\begin{corollary}
Let $\mc{A},\mc{B}\in\mb{C}^{I_{1\ldots N}\times I_{1\ldots N}}$. Then $w(\mc{A}\n\mc{B})\leq 4w(\mc{A})w(\mc{B})$.
\end{corollary}

Next, we recall two results on the determinant of tensors due to Liang {\it et al.} \cite{liang2019.1}.
\begin{theorem}[Theorem 3.16, \cite{liang2019.1}]\leavevmode\\
Let $\mc{A},\mc{B}\in \mb{R}^{I_{1\ldots N}\times I_{1\ldots N}}$ be two tensors. Then $\det(\mc{A}\n\mc{B})=\det(\mc{A})\det(\mc{B})$.
\end{theorem}
\begin{theorem}[Theorem 4.10, \cite{liang2019.1}]\leavevmode\\
Let $\mc{A} = (a_{i_{1}\ldots i_{N}j_{1}\ldots j_{N}}) \in \mb{C}^{I_{1\ldots N}\times I_{1\ldots N}}$ be a given tensor, then $\det (\mc{A}) = \prod_{i_{1},\ldots,i_{N}} \lambda_{i_{1}\ldots i_{N}}$.

\end{theorem}
For two tensors $\mc{A},\mc{B}\in \mb{C}^{I_{1\ldots N}\times I_{1\ldots N}}$, define $|\mc{A}|:=(\mc{A}^{H}\n\mc{A})^{1/2}$ and $\mc{A}\geq \mc{B}$ means $\mc{A}-\mc{B}$ is a positive definite tensor. The following result provides a sufficient condition for unitarity of a tensor. 

\begin{lemma}\label{lem:untry}
Let $\mc{A}\in\mb{C}^{I_{1\ldots N}\times I_{1\ldots N}}$ be an invertible tensor such that $w(\mc{A})\leq 1$
and $|\mc{A}|\geq \mc{I}.$ Then $\mc{A}$ is unitary.
\end{lemma}


Lemma \ref{lem:untry} is helpful in the next theorem to prove a new necessary and sufficient condition for a unitary tensor. 

\begin{theorem}\label{thm:untry1}
$\mc{A}\in\mb{C}^{I_{1\ldots N}\times I_{1\ldots N}}$ is an invertible tensor such that $w(\mc{A})\leq 1$ and $w(\mc{A}^{-1})\leq 1$ if, and only if, $\mc{A}$ is unitary. 
\end{theorem}
\begin{proof}
Let $\mc{A}=\mc{U}\n|\mc{A}|$ be the polar decomposition of $\mc{A}.$ Then $(\mc{A}^{-1})^{H}=((\mc{U}\n|\mc{A}|)^{-1})^{H}=(|\mc{A}|^{-1}\n\mc{U}^{-1})^{H}=\mc{U}\n(|\mc{A}|^{-1})^{H}=\mc{U}\n|\mc{A}|^{-1}.$ Since $w(\mc{A}^{-1})\leq 1$, so for any unit tensor $\mc{X}\in \mb{C}^{I_{1\ldots N}}$, we have
$|\left\langle\mc{U}\n|\mc{A}|^{-1}\n\mc{X},\mc{X}\right\rangle|=|\left\langle(\mc{A}^{-1})^{H}\n\mc{X},\mc{X}\right\rangle|=|\left\langle\mc{A}^{-1}\n\mc{X},\mc{X}\right\rangle|\leq 1.$
Let $\mc{B}:=\mc{U}\n\frac{|\mc{A}|+|\mc{A}|^{-1}}{2}.$ Here $|\mc{B}|=\frac{|\mc{A}|+|\mc{A}|^{-1}}{2}.$ Now, 
\begin{align*}
    |\left\langle\mc{B}\n\mc{X},\mc{X}\right\rangle|&=\left|\left\langle\left(\mc{U}\n\frac{|\mc{A}|+|\mc{A}|^{-1}}{2}\right)\n\mc{X},\mc{X}\right\rangle\right|\\
    &=\frac{1}{2} |\left\langle\mc{U}\n|\mc{A}|\n\mc{X},\mc{X}\right\rangle+\left\langle \mc{U}\n|\mc{A}|^{-1}\n\mc{X},\mc{X}\right\rangle|\\
    &= \frac{1}{2} |\left\langle\mc{A}\n\mc{X},\mc{X}\right\rangle+\left\langle\mc{A}^{-1}\n\mc{X},\mc{X}\right\rangle|\\
    &\leq 1.
\end{align*}
Thus $w(\mc{B})\leq 1.$ Again, since $|\mc{A}|>0$ and $|\mc{A}|^{-1}>0$, so $|\mc{A}|+|\mc{A}|^{-1}-2\mc{I}=(|\mc{A}|^{1/2}-|\mc{A}|^{-1/2})^2\geq 0.$ Hence $\frac{|\mc{A}|+|\mc{A}|^{-1}}{2}\geq \mc{I}$, i.e., $|\mc{B}|\geq \mc{I}.$ Thus $\mc{B}$ is unitary due to Lemma \ref{lem:untry} and hence $|\mc{B}|=\mc{I}$, i.e., $|\mc{A}|+|\mc{A}|^{-1}=2\mc{I}.$ This leads to $|\mc{A}|=\mc{I}.$ Thus $\mc{A}$ is unitary. Converse part is obvious as $w(\mc{A})\leq \|\mc{A}\|$ by Theorem \ref{thm:bscnrineq}.
\end{proof}

\section{Numerical range of Moore--Penrose inverse of an even-order square tensor}\label{sec:nrmpitnsr}
In this section, we concentrate on the numerical range of Moore--Penrose inverse of an even-order square tensor and investigate how it relates with the numerical range of original tensor. The first result of this section confirms that both the tensors $\mc{A}$ and its Moore--Penrose inverse $\mc{A}^{\dg}$ are Hermitian or normal simultaneously. 

\begin{theorem}
Let $\mc{A}\in \mb{C}^{I_{1\ldots N}\times I_{1\ldots N}}$. Then $\mc{A}$ is normal (resp. Hermitian) if, and only if, $\mc{A}^{\dg}$ is normal (resp. Hermitian).
\end{theorem}

In general, if $\lambda\neq 0$ is an eigenvalue of a tensor $\mc{A}\in\mb{C}^{I_{1\ldots N}\times I_{1\ldots N}}$, then $1/\lambda$ may not be an eigenvalue of the tensor $\mc{A}^{\dg}$. However, if $\mc{A}$ is normal then $\lambda\neq 0$ is an eigenvalue of a tensor $\mc{A}$ implies $1/\lambda$ is an eigenvalue of the tensor $\mc{A}^{\dg}$. While if $0$ is an eigenvalue of a tensor $\mc{A}$, then $0$ is always an eigenvalue of $\mc{A}^{\dg}$ for any tensor $\mc{A}\in\mb{C}^{I_{1\ldots N}\times I_{1\ldots N}}$. This is shown in the next result.
\begin{theorem}\label{thm:egnvlrln}
    Let $\mc{A}\in \mb{C}^{I_{1\ldots N}\times I_{1\ldots N}}$. Then 
    \begin{enumerate}[i)]
        \item $0\in\sigma{(\mc{A})}$ if, and only if, $0\in\sigma(\mc{A}^{\dg})$;
        \item If $\mc{A}$ is normal and $\lambda\neq 0$, then $\lambda\in\sigma(\mc{A})$ if, and only if, $1/\lambda\in\sigma(\mc{A}^{\dg})$.
    \end{enumerate}
    \end{theorem}

Note that the result in Theorem \ref{thm:egnvlrln} ii) does not hold, if $\mc{A}$ is not a normal tensor.
\begin{example}
Consider the tensor $\mc{A}\in \mb{C}^{3\times2\times 3\times 2}$ as below
\begin{center}
{\scriptsize\begin{tabular}{cc|cc|cc|cc|cc|cc}
\hline
     \multicolumn{2}{c}{$\mc{A}(:,:,1,1)$} & \multicolumn{2}{c}{$\mc{A}(:,:,2,1)$} & \multicolumn{2}{c}{$\mc{A}(:,:,3,1)$} & \multicolumn{2}{c}{$\mc{A}(:,:,1,2)$} & \multicolumn{2}{c}{$\mc{A}(:,:,2,2)$} & \multicolumn{2}{c}{$\mc{A}(:,:,3,2)$}  \\
\hline
\hline
1&0&1&0&1&0&1&0&1&0&1&0\\
0&0&0&0&0&0&0&0&0&0&0&0\\
0&0&0&0&0&0&0&0&0&0&0&0\\
\hline
\end{tabular}}.
\end{center}
Then it's Moore--Penrose inverse $\mc{A}^{\dg}\in\mb{C}^{3\times2\times 3\times2}$ becomes
\begin{center}
{\scriptsize\begin{tabular}{cc|cc|cc|cc|cc|cc}
\hline
     \multicolumn{2}{c}{$\mc{A}^{\dg}(:,:,1,1)$} & \multicolumn{2}{c}{$\mc{A}^{\dg}(:,:,2,1)$} & \multicolumn{2}{c}{$\mc{A}^{\dg}(:,:,3,1)$} & \multicolumn{2}{c}{$\mc{A}^{\dg}(:,:,1,2)$} & \multicolumn{2}{c}{$\mc{A}^{\dg}(:,:,2,2)$} & \multicolumn{2}{c}{$\mc{A}^{\dg}(:,:,3,2)$}  \\
\hline
\hline
0.1667&0.1667&0&0&0&0&0&0&0&0&0&0\\
0.1667&0.1667&0&0&0&0&0&0&0&0&0&0\\
0.1667&0.1667&0&0&0&0&0&0&0&0&0&0\\
\hline
\end{tabular}}.
\end{center}
Now, $\sigma(\mc{A})=\{0,1\}$ and $\sigma(\mc{A}^{\dg})=\{0,0.1667\}$. Observe that here $\mc{A}\n\mc{A}^{\dg}\neq\mc{A}^{\dg}\n\mc{A}$ and $1$ is an eigenvalue of $\mc{A}$ while $1$ is not an eigenvalue of $\mc{A}^{\dg}$.
\end{example}
Next result confirms that $0$ contained in the numerical range of a tensor if, and only if, it is contained in numerical range of its Moore--Penrose inverse.
\begin{theorem}\label{thm:nrmpinr}
Let $\mc{A}\in \mb{C}^{I_{1\ldots N}\times I_{1\ldots N}}$. Then $0\in W(\mc{A})$ if, and only if, $0\in W(\mc{A}^{\dg})$.
\end{theorem}

The next result shows that $W(\mc{A})=W(\mc{A}^{H})$ is sufficient to confirm that set $W(\mc{A})$ and  $\alpha^{2}W(\mc{A}^{\dg})$ are disjoint, where `$\alpha$' is a singular value of $\mc{A}$ (positive square roots of eigenvalues of $\mc{A}^{H}\n\mc{A}$). 
\begin{theorem}
    Let $\mc{A}\in \mb{C}^{I_{1\ldots N}\times I_{1\ldots N}}$ such that $W(\mc{A})=W(\mc{A}^{H})$. Then 
    \begin{equation*}
        W(\mc{A})\bigcap \alpha^{2}W(\mc{A}^{\dg})\neq \emptyset
    \end{equation*}
where `$\alpha$' is a singular value of $\mc{A}$.
    \end{theorem}

We want to bring the readers attention that if $W(\mc{A})=W(\mc{A}^{H})$ is omitted from the above result, then the result may not hold.
\begin{example}
Consider a tensor $\mc{A}\in \mb{C}^{3\times 2\times 3\times 2}$ such that 
\begin{center}
    \begin{tabular}{cc|cc|cc|cc|cc|cc}
    \hline
    \multicolumn{2}{c}{$\mc{A}(:,:,1,1)$} & \multicolumn{2}{c}{$\mc{A}(:,:,2,1)$} &
    \multicolumn{2}{c}{$\mc{A}(:,:,3,1)$} &
    \multicolumn{2}{c}{$\mc{A}(:,:,1,2)$} &
    \multicolumn{2}{c}{$\mc{A}(:,:,2,2)$} &
    \multicolumn{2}{c}{$\mc{A}(:,:,3,2)$}\\
    \hline
    $1+i$ & $0$ & $0$ & $0$ & $0$ & $0$ & $0$& $4$& $0$ & $0$ & $0$ & $0$ \\
    $0$ & $0$ & $i$ & $0$ & $0$ & $0$ & $0$ & $0$ & $0$ & $5+i$ &$0$ & $0$ \\
    $0$ & $0$ & $0$ & $0$ & $3+i$ & $0$ & $0$ & $0$ & $0$ & $0$ & $0$ & $6+i$ \\
    \hline
\end{tabular}.
\end{center}
Then, the conjugate transpose of $\mc{A}$, $\mc{A}^{H}\in \mb{C}^{3\times 2\times 3\times 2}$, is 
\begin{center}
    \begin{tabular}{cc|cc|cc|cc|cc|cc}
    \hline
    \multicolumn{2}{c}{$\mc{A}^{H}(:,:,1,1)$} & \multicolumn{2}{c}{$\mc{A}^{H}(:,:,2,1)$} &
    \multicolumn{2}{c}{$\mc{A}^{H}(:,:,3,1)$} &
    \multicolumn{2}{c}{$\mc{A}^{H}(:,:,1,2)$} &
    \multicolumn{2}{c}{$\mc{A}^{H}(:,:,2,2)$} &
    \multicolumn{2}{c}{$\mc{A}^{H}(:,:,3,2)$}\\
    \hline
    $1-i$ & $0$ & $0$ & $0$ & $0$ & $0$ & $0$& $4$ & $0$ & $0$ & $0 $& $0$ \\
    $0$ & $0$ & $-i$ & $0$ & $0$ & $0$ & $0$ & $0$ & $0$ & $5-i$ &$0$ & $0$ \\
    $0$ & $0$ & $0$ & $0$ & $3-i$ & $0$ & $0$ & $0$ & $0$ & $0$ & $0$ & $6-i$ \\
    \hline
\end{tabular}.
\end{center}
Thus, $W(\mc{A})\neq W(\mc{A}^{H})$. The set of singular values of the tensor $\mc{A}$ is $\{1,\sqrt{2}, \sqrt{10}, 4, \sqrt{26},\sqrt{37}\}$. Now, the Moore--Penrose inverse of $\mc{A}$, $\mc{A}^{\dg}\in\mb{C}^{3\times 2\times 3\times 2}$, is
\begin{center}
    \begin{tabular}{cc|cc|cc|cc|cc|cc}
    \hline
    \multicolumn{2}{c}{$\mc{A}^{\dg}(:,:,1,1)$} & \multicolumn{2}{c}{$\mc{A}^{\dg}(:,:,2,1)$} &
    \multicolumn{2}{c}{$\mc{A}^{\dg}(:,:,3,1)$} &
    \multicolumn{2}{c}{$\mc{A}^{\dg}(:,:,1,2)$} &
    \multicolumn{2}{c}{$\mc{A}^{\dg}(:,:,2,2)$} &
    \multicolumn{2}{c}{$\mc{A}^{\dg}(:,:,3,2)$}\\
    \hline
    $0.5-0.5i$ & $0$ & $0$ & $0$ & $0$ & $0$ & $0$& $0.25$ & $0$ & $0$ & $0 $& $0$ \\
    $0$ & $0$ & $-i$ & $0$ & $0$ & $0$ & $0$ & $0$ & $0$ & $(5-i)/26$ &$0$ & $0$ \\
    $0$ & $0$ & $0$ & $0$ & $0.3-0.1i$ & $0$ & $0$ & $0$ & $0$ & $0$ & $0$ & $(6-i)/37$ \\
    \hline
\end{tabular}.
\end{center}
From Figure \ref{fig:nrintersection} it is clear that $W(\mc{A})\bigcap \alpha^{2}W(\mc{A}^{\dg})=\emptyset$ for $\alpha\in \{1, \sqrt{2}, \sqrt{10}, \sqrt{26},\sqrt{37}\}$ and when $\alpha=4$ we have $W(\mc{A})\bigcap \alpha^{2}W(\mc{A}^{\dg})=\{4\}\neq\emptyset$.
\begin{figure}[h!]
     \centering
     \begin{subfigure}[b]{0.4\textwidth}
         \centering
         \includegraphics[width=7cm,height=3.5cm]{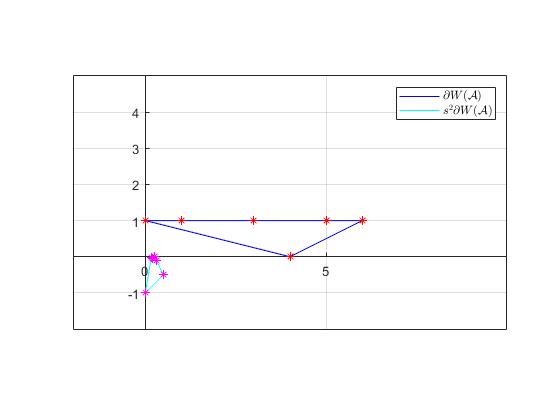}
         \caption{$\alpha^{2}=1$}
         \label{fig:s=1}
     \end{subfigure}
     \hfill
     \begin{subfigure}[b]{0.4\textwidth}
         \centering
         \includegraphics[width=7cm,height=3.5cm]{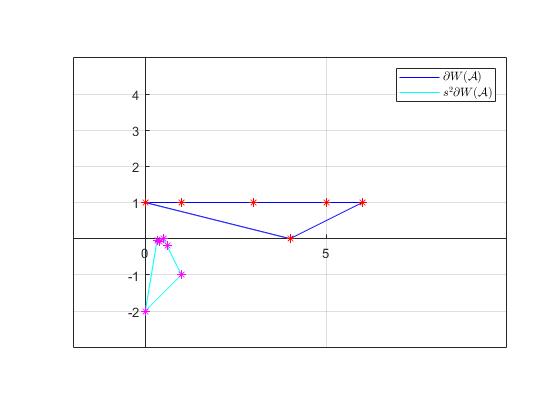}
         \caption{$\alpha^{2}=2$}
         \label{fig:s=2}
     \end{subfigure}
     \hfill
     \begin{subfigure}[b]{0.4\textwidth}
         \centering
         \includegraphics[width=7cm,height=3.5cm]{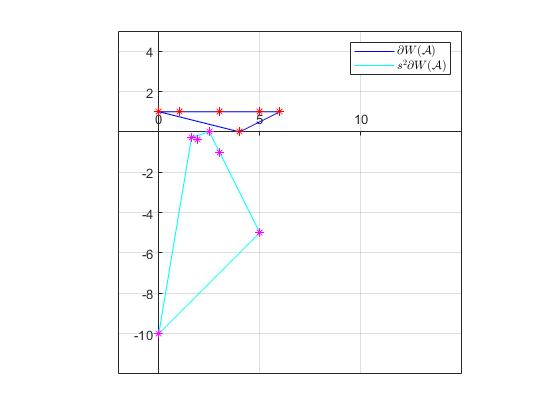}
         \caption{$\alpha^{2}=10$}
         \label{fig:s=10}
     \end{subfigure}
     \hfill
     \begin{subfigure}[b]{0.4\textwidth}
         \centering
         \includegraphics[width=7cm,height=3.5cm]{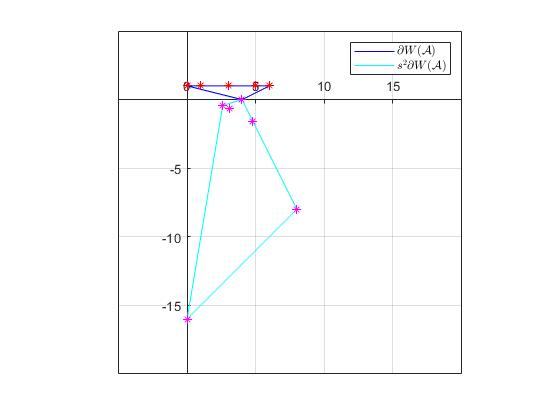}
         \caption{$\alpha^{2}=16$}
         \label{fig:s=16}
     \end{subfigure}
     \begin{subfigure}[b]{0.4\textwidth}
         \centering
    \includegraphics[width=7cm,height=3.5cm]{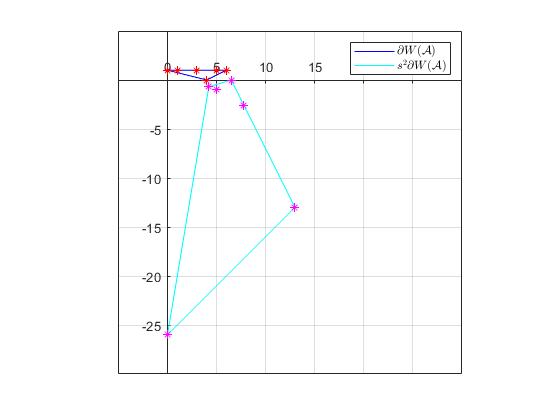}
         \caption{$\alpha^{2}=26$}
         \label{fig:s=26}
     \end{subfigure}
     \begin{subfigure}[b]{0.4\textwidth}
         \centering
    \includegraphics[width=7cm,height=3.5cm]{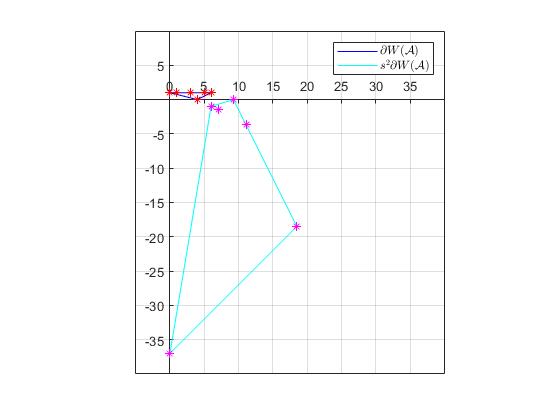}
         \caption{$\alpha^{2}=37$}
         \label{fig:s=37}
     \end{subfigure}
        \caption{Boundaries of numerical ranges of $W(\mc{A})$ and $\alpha^{2}W(\mc{A}^{\dg})$}
        \label{fig:nrintersection}
\end{figure}
\end{example}
Next we establish a relation between $\sigma(\mc{A})$, $W(\mc{A})$ and $\dfrac{1}{W(\mc{A})}$. Recall that a tensor $\mc{A}$ is called an EP-tensor if it satisfies $\mc{A}\n\mc{A}^{\dg}=\mc{A}^{\dg}\n\mc{A}$. 
\begin{theorem}
Let $\mc{A}\in \mb{C}^{I_{1\ldots N}\times I_{1\ldots N}}$ be an EP-tensor. Then, 
\begin{equation*}
    \sigma(\mc{A})\subset W(\mc{A})\bigcap \dfrac{1}{W(\mc{A}^{\dg})}.
\end{equation*}
\end{theorem}

Let $\mc{A},\mc{B}\in\mb{C}^{I_{1\ldots N}\times I_{1\ldots N}}$, then we define a block tensor \cite{behera2017}
\begin{equation}
    \mc{A}\oplus \mc{B}=\begin{bmatrix}\mc{A}&\mc{O}\\\mc{O}&\mc{B} \end{bmatrix}\in\mb{C}^{J_{1\ldots N}\times J_{1\ldots N}},
\end{equation}
where $\mc{O}\in\mb{C}^{I_{1\ldots N}\times I_{1\ldots N}}$ and $J_{i}=2I_{i}$ where $i\in\{1,2,\ldots,N\}$.
Next result provides a procedure to calculate Moore--Penrose inverse of a special tensor.

\begin{theorem}
Let $\{\mc{U}_{1},~\mc{U}_{2},\ldots,~\mc{U}_{r}\}$ and $\{\mc{V}_{1},~\mc{V}_{2},\ldots,~\mc{V}_{r}\}$ be two orthonormal subsets of $\mb{C}^{I_{1\ldots N}}$. If $\mc{A}=\mc{U}_{1}\n\mc{V}_{1}^{H}+\mc{U}_{2}\n\mc{V}_{2}^{H}+\ldots+\mc{U}_{r}\n\mc{V}_{r}^{H}$, then $\mc{A}^{\dg}=\mc{V}_{1}\n\mc{U}_{1}^{H}+\mc{V}_{2}\n\mc{U}_{2}^{H}+\ldots+\mc{V}_{r}\n\mc{U}_{r}^{H}$, and $W(\mc{A}^{\dg})=W(\mc{A}^{H})$.
\end{theorem}

The following result gives an inequality between product of spectral norm of a tensor with its Moore--Penrose inverse and their product of numerical radius.
\begin{theorem}\label{numrclrdsineqa}
Let $\mc{O}\neq \mc{A}\in \mb{C}^{I_{1\ldots N}\times I_{1\ldots N}}$. Then, for the spectral norm $\|\cdot\|$,
\begin{equation*}
    1\leq \|\mc{A}\|\|\mc{A}^{\dg}\|\leq 4w(\mc{A})w(\mc{A}^{\dg}).
\end{equation*}
\end{theorem}

Next, we provide an example to verify the above inequality.
\begin{example}
Consider a tensor $\mc{A}\in \mb{R}^{2\times 2\times 2\times 2}$ such that
\begin{center}
    \begin{tabular}{cc|cc|cc|cc}
    \hline
        \multicolumn{2}{c}{$\mc{A}(:,:,1,1)$} & \multicolumn{2}{c}{$\mc{A}(:,:,2,1)$} & \multicolumn{2}{c}{$\mc{A}(:,:,1,2)$} & \multicolumn{2}{c}{$\mc{A}(:,:,2,2)$}  \\
        \hline
        $2$ & $5$ & $7$ & $9$ & $0$& $11$ & $1$& $-1$\\
        $-5$ & $0$ & $5$ & $7$ & $4$ & $8$ & $9$ & $2$\\
        \hline
    \end{tabular}.
\end{center}
Then, the Moore--Penrose inverse of $\mc{A}$, $\mc{A}^{\dg}\in\mb{R}^{2\times 2\times 2\times 2}$, is
\begin{center}
    \begin{tabular}{cc|cc|cc|cc}
    \hline
        \multicolumn{2}{c}{$\mc{A}(:,:,1,1)$} & \multicolumn{2}{c}{$\mc{A}(:,:,2,1)$} & \multicolumn{2}{c}{$\mc{A}(:,:,1,2)$} & \multicolumn{2}{c}{$\mc{A}(:,:,2,2)$}  \\
        \hline
        $-0.0044$ & $-0.1223$ & $0.1790$ & $0.0131$ & $0.3808$& $0.0620$ & $-0.6131$& $0.0332$\\
        $0.1485$ & $-0.0306$ & $-0.0873$ & $0.2533$ & $-0.1467$ & $0.2655$ & $0.2454$ & $-0.4917$\\
        \hline
    \end{tabular}.
\end{center}
Here, $\|\mc{A}\|=19.9331,\|\mc{A}^{\dg}\|=1.0076, w(\mc{A})=18.9853$ and $w(\mc{A}^{\dg})=0.8253$. Thus, $ 1\leq \|\mc{A}\|\|\mc{A}^{\dg}\|\leq 4w(\mc{A})w(\mc{A}^{\dg})$ holds.
\end{example}

\bibliographystyle{amsplain}

\end{document}